\documentclass[12pt,reqno]{amsart}
\usepackage{hyperref}
\usepackage{amsthm, amsmath, amscd, amssymb,centernot}
\usepackage{amsmath}
\usepackage{amssymb}
\usepackage[all]{xy}
\usepackage{ulem}
\usepackage{xcolor}

\newtheorem{theorem}{Theorem}[section]
\newtheorem{proposition}[theorem]{Proposition}
\newtheorem{lemma}[theorem]{Lemma}
\newtheorem{corollary}[theorem]{Corollary}

\theoremstyle{definition}

\newtheorem{remark}[theorem]{Remark}
\newtheorem{example}[theorem]{Example}

\numberwithin{equation}{section}

\newcommand{\C}{\mathbb C}

\newcommand{\R}{\mathbb R}
\newcommand{\Z}{\mathbb Z}

\begin{document}

\baselineskip=15pt

\title[Seshadri constants on Bott towers]{Seshadri constants on Bott towers}

\author[I. Biswas]{Indranil Biswas}

\address{School of Mathematics, Tata Institute of Fundamental
Research, Homi Bhabha Road, Mumbai 400005, India}

\email{indranil@math.tifr.res.in}

\author[J. Dasgupta]{Jyoti Dasgupta}

\address{Indian Institute of Science Education and Research, Pune, 
Dr. Homi Bhabha Road,
Pashan, Pune 411 008, India}

\email{jdasgupta.maths@gmail.com}

\author[K. Hanumanthu]{Krishna Hanumanthu}

\address{Chennai Mathematical Institute, H1 SIPCOT IT Park, Siruseri, Kelambakkam 603103, 
India}
\email{krishna@cmi.ac.in}

\author[B. Khan]{Bivas Khan}

\address{Indian Institute of Science Education and Research, Pune, 
Dr. Homi Bhabha Road,
Pashan, Pune 411 008, India}

\email{bivaskhan10@gmail.com}



\date{}

\begin{abstract}
For a positive integer $n$, let $X_n \to X_{n-1} \to \ldots \to X_2
\to X_1 \to X_0$ be a Bott tower of height $n$, and let $L$ be a nef line bundle on
$X_n$. 
We compute Seshadri constants $\varepsilon(X_n,L,x)$ of $L$ at any point $x \in
X_n$ under some conditions. 
\end{abstract}

\maketitle

\section{Introduction}\label{intro}

Seshadri constants of line bundles reflect their local positivity.
Soon after Demailly introduced Seshadri constants in \cite{D},
there has been extensive work on them; they have turned out to
be important invariants. Let us briefly recall their
definition. 

Let $X$ be a complex projective variety, and let $L$ be a nef line bundle on
$X$. For a point $x \in X$, the \textit{Seshadri constant} of $L$ at $x$,
denoted by $\varepsilon(X,L,x)$, is defined
to be
$$\varepsilon(X,L,x)\,:=\, \inf\limits_{\substack{x \in C}} \frac{L\cdot
C}{{\rm mult}_{x}C}\, ,$$
where the infimum is taken over all closed curves $C\,\subset\, X$ passing through $x$. Here 
$L\cdot C$ denotes the intersection number while ${\rm mult}_x C$ denotes the 
multiplicity of the curve $C$ at $x$. To compute the Seshadri constant
$\varepsilon(X,L,x)$, it suffices to take only irreducible and reduced curves
$C$ in the above definition. 

There is another formulation of Seshadri constants which is often useful. 
Let $\pi\,:\, \widetilde{X}\,\longrightarrow\, X$ be the blow up of $X$ at $x$ and let $E$ denote the exceptional
divisor. Then $$\epsilon(X,L,x)\, =\,
{\rm sup}~\{\lambda \ge 0 ~|~ \pi^{\ast}(L)-\lambda E ~\rm{is~ nef}\}.$$

The Seshadri's criterion for ampleness of a line bundle says
that $L$ is ample if and only if $\varepsilon(X,L,x) \,>\, 0$ for all
$x\,\in\, X$.  Indeed, if $L$ is ample then $mL$ is very ample for some positive integer $m$. Then it
is easy to check that $\varepsilon(X,L,x) \,\geq\, \frac{1}{m}$ for all $x$. For the converse, we use
the Nakai-Moishezon criterion to verify ampleness of $L$. By induction on the dimension of $X$, we can 
assume that $L^i\cdot Y > 0$ for every closed subvariety $Y \,\subset\, X$ of dimension $i \,<\, n\,:=\, \text{dim}(X)$.  
It remains to prove that $L^n > 0$, where $L^n$ denotes the 
top self-intersection number of $L$. For this, let $\pi\,:\, \widetilde{X}\,\longrightarrow\, X$ be the blow up of $X$ at a 
smooth point $x \in X$ as above. Then $\pi^{\ast}L-\varepsilon(X,L,x)E$ is nef.
So  $(\pi^{\ast}L-\varepsilon(X,L,x)E)^n \ge 0$ which implies 
$L^n \ge \varepsilon(X, L,x)^n > 0$. See the proof of \cite[Chapter 1, Page 37, Theorem 7.1]{Har} for more details. 

Let $L$ be an ample line bundle on a projective variety $X$. It is easy to see that 
$$\varepsilon(X,L,x) \,\le\, \sqrt[n]{L^n}\,.$$  One then defines 
$$\varepsilon(X,L,1) \,:=\, \sup\limits_{x\in X}
\varepsilon(X,L,x)\, .$$ 
Similarly, we have 
$$\varepsilon(X,L) \,:=\, \inf\limits_{\substack{x \in X}} \varepsilon(X,L,x)\, .$$

Seshadri constants have many interesting applications and they are now
the focus of a very active area of research. 
Some of the guiding problems on Seshadri constants involve computing
Seshadri constants, giving
bounds on them, checking if they
are irrational, and interpolation problems. Computing Seshadri
constants is frequently very difficult and usually it is only possible to give some bounds. In
some special cases, however, it is possible to compute
them exactly. In this paper we compute Seshadri constants of line bundles on Bott towers
at all points. 

Most of the existing work on Seshadri constants has been in the case
of surfaces. Among the few cases in higher dimensions where Seshadri
constants have been studied are abelian varieties (for example, see \cite{Na, La, Ba,
Deb}), toric varieties (for example, see \cite{DiRocco, HMP, It1,
It2}), Fano varieties (for example, see \cite{BS,
LZ}), and Grassmann bundles over curves (\cite{BHNN}). For a survey of
research around Seshadri constants, see \cite{B}. 

In this paper, we study Seshadri constants for line bundles on Bott
towers. We recall that Bott towers are special classes of toric varieties constructed
iteratively as projective bundles of rank two vector bundles starting
with the projective line $\mathbb{P}^1$. One can view them as
a generalization of \textit{Hirzebruch surfaces}, which are
geometrically ruled surfaces over $\mathbb{P}^1$. See Section
\ref{bott} for more details on Bott towers. 

Seshadri constants of line bundles on Hirzebruch surfaces have been
computed (see \cite{Sy,Ga,HM}). In this paper we generalize this
computation to Bott towers under some conditions. Our main result (Theorem \ref{main})
computes the Seshadri constants for an arbitrary nef line bundle on
a Bott tower at any point.

As noted above, Seshadri constants for line bundles on
toric varieties have been studied by various authors. But for an
arbitrary toric variety, Seshadri constants have been computed only
for some classes of points, such as torus fixed points or points on the
torus; see Remark \ref{hmp} and Remark \ref{ito}. 
In this paper, using the additional structure of a Bott tower,
we compute Seshadri constants at arbitrary points. 

In Section \ref{bott}, we recall the construction 
of Bott towers and prove some properties which are used in Section \ref{sc}. 
In Section \ref{sc}, we prove our main theorem
computing the Seshadri constants of nef
line bundles on Bott towers. In Subsection \ref{examples}, we include
some remarks comparing our results with existing results in the
literature and give examples illustrating our results. 

\subsection*{Notation} We work over the field of complex
numbers. We write $D_1 \,\sim_{\text{lin}}\, D_2$ (respectively, $D_1
\,\equiv\, D_2$) if the divisors $D_1, D_2$ are
linearly equivalent (respectively, numerically equivalent). When the
variety $X$ is clear from the context, the Seshadri constant
$\varepsilon(X,L,x)$ is denoted simply by $\varepsilon(L,x)$. 

\section*{Acknowledgements}

We thank the anonymous referee for a careful reading 
and numerous helpful suggestions which improved the paper. 
The first author is supported by a J. C. Bose Fellowship.
The third author is partially supported by DST SERB MATRICS grant MTR/2017/000243
and also a grant from Infosys Foundation. The last author is supported by NBHM DAE post-doctoral fellowship. This research was 
supported in part by the International Centre for Theoretical Sciences
(ICTS) during visits for participating in the following 
programs: Topics in Birational Geometry (Code: ICTS/TBG2020/01)
and Moduli of bundles and related structures (Code: ICTS/MBRS2020/02).

\section{Bott towers}\label{bott}

In this section, after recalling the construction of Bott towers along with some results about them,
we prove some results about Bott towers that will be used in the computation
of Seshadri constants.

Bott towers are a particular class of nonsingular projective toric
varieties. They were constructed by Grossberg and Karshon (see
\cite{GK}). Grossberg and Karshon
have also shown that Bott towers are degenerations of Bott-Samelson varieties, which are desingularizations of Schubert varieties. 

For an integer $n \, \ge \, 0$, a \textit{Bott tower of height $n$}
\begin{equation}\label{bn}
X_n \,\longrightarrow\, X_{n-1}\,\longrightarrow\, \ldots \,\longrightarrow\, X_2\,
\longrightarrow\, X_1 \,\longrightarrow\, X_0\,=\,\{\text{point} \}
\end{equation} 
is defined inductively as an iterated ${\mathbb P}^1$--bundle so that at the 
$k$-th stage of the tower, $X_k$ is of the form
$\mathbb{P}(\mathcal{O}_{X_{k-1}}\oplus \mathcal{L}$) 
for a line bundle $\mathcal{L}$ over $X_{k-1}$.
So $X_1$ is isomorphic to \(\mathbb{P}^1\), $X_2$ is
a Hirzebruch surface and so on. A classical
example is the product of projective lines, 
which arises when the line bundle $\mathcal{L}$ is
trivial at every stage. 

We call any stage $X_i$ of the tower $X_n$ in \eqref{bn} also a Bott tower.

\subsection{Fan structure of a Bott tower}

The multiplicative group $\C\setminus\{0\}$ will be denoted by $\C^*$.
Let $T \,\cong\, \left(\C^*\right)^n$ be an algebraic torus. 
Define its character lattice $$M\,:=\,\text{Hom}(T, \,\C^*) \,\cong\, \Z^n$$ and
the dual lattice $N\,:=\,\text{Hom}_{\Z}(M, \Z)$. Let $\Delta_n$ be a
fan in $N_{\R}\,:=\,N \otimes_{\Z} \R$ which defines the toric variety
$X_n$ under the action of the torus 
$T$. The set of edges of $\Delta_n$ will be denoted by
$\Delta_n(1)$. Let $e_1,\,\cdots,\, e_n$ be the standard basis for $\mathbb{R}^n$. 
Consider the following vectors:
\begin{equation}\label{ci}
\begin{split}
& v_1\,=\,e_1,\, \cdots,\, v_n\,=\,e_n, \\
& v_{n+1}\,=\,-e_1+c_{1, 2} e_2 + \ldots + c_{1, n} e_n,\\
& \vdots\\
& v_{n+i}\,=\,-e_i+c_{i, i+1} e_{i+1} + \ldots + c_{i, n} e_n, 1 \leq i <n,\\
&v_{2n}\,=\,-e_n.
\end{split}
\end{equation}
The fan $\Delta_n$ of $X_n$ is complete, and it consists of these $2n$
edges and $2^n$ maximal cones of dimension $n$ generated by these 
edges such that no cone contains both the edges $v_i$ and $v_{n+i}$
for $i\,=\,1,\, \cdots,\, n$. It
follows that any \(k\)-th stage Bott tower arises from a collection
of integers \(\{c_{i,j} \}_{1 \leq i< j \leq n}\) as in \eqref{ci}. These integers are
called the \textit{Bott numbers} of the given Bott tower. 
In this paper we will restrict our attention to the case when the
Bott numbers $\{c_{i, j}\}_{\{1\leq i < j \leq n \}}$ are all positive
integers. 

\subsection{Picard group of a Bott tower}

The following is recalled from \cite[Section 2.2]{KD}.

Let $D_i$ denote the invariant prime divisor corresponding to the 
edge $v_{n+i}$, and let $D'_i$ denote the invariant prime divisor corresponding to the edge $v_i$
for \(i\,=\,1,\, \cdots,\, n\). We have the following relations:
\begin{equation}\label{linequiv}
D'_1 \,\sim_{\text{lin}}\, D_1,\
D'_i \,\sim_{\text{lin}}\, D_{i}-c_{1, i}D_{1}-\ldots-c_{i-1, i}D_{i-1}
\end{equation}
for $i\,=\,2,\, \cdots,\, n$.
The Picard group of the Bott tower is $$\text{Pic}(X_n)\,=\,\Z
D_{1} \oplus \,\cdots\, \oplus \Z D_{n}\, .$$ 
If $L$ is a line bundle on $X_n$ which is numerically equivalent to
$a_1 D_1 + \ldots + a_n D_n$ for some integers $a_1,\,\cdots,\, a_n$, then we write \(L
\,\equiv\, (a_1,\, \cdots,\, a_n)\).

Let $D\,=\,\sum_{i=1}^{k}a_iD_{i}$ be a Cartier divisor on $X_n$. Then $D$
is ample (respectively, nef) if and only if $a_i \,> 0 \text{ (respectively, $a_i \ \geq \ 0$) }$ for all 
$i\,=\,1,\, \cdots, \,n$ (see \cite[Theorem 3.1.1, Corollary 3.1.2]{KD}).
	
\subsection{Quotient construction of a Bott tower}

We recall the quotient construction of Bott tower from \cite[Theorem
7.8.7]{BP}. The Bott tower \(X_n\) can be obtained as the 
quotient $U_{n}/\!\!/G_{n}$ of
\[U_n\,=\,\{(z_1,w_1, \cdots, z_n, w_n) \,\in\, \mathbb{C}^{2n} \,\mid\
|z_i|^2+|w_i|^2 \neq 0,\ 1 \,\leq\, i \,\leq\, n\} \,\cong\, (\mathbb{C}^2 \setminus {0})^n\]
for the action of the group
\[G_n=\{(t_{\rho})_{\rho \in \Delta_n(1)}
\in (\mathbb{C}^*)^{\Delta_n(1)} \mid \prod\limits_{\rho \in
\Delta_n(1)} t_{\rho}^{\langle u_i, v_{\rho} \rangle}=1\} \cong
(\mathbb{C}^*)^n\, ,\] 
where \(u_1,\, \cdots, \,u_n\) is a basis of \(M\). More explicitly, 
the inclusion \((\mathbb{C}^*)^n \,\hookrightarrow\, (\mathbb{C}^*)^{2n}\) is given by 
\begin{equation*}
(t_1,\, \cdots,\, t_n) \,\longmapsto\,
(t_1,\, t_1,\, t_1^{-c_{1,2}} t_2,\, t_2, \,\cdots,\, t_1^{-c_{1,n}} t_2^{-c_{2,n}}
\cdots t_{n-1}^{c_{n-1, n}} t_n,\,t_n ).
\end{equation*}
A point of \(X_n\) is denoted by the equivalence class \([z_1:w_1: \ldots: z_n: w_n]\).
Note that \(D'_i\) (respectively, \(D_i\)) is just the vanishing locus of the coordinate \(z_i\) (respectively,
\(w_i\)), i.e.,
\begin{equation}\label{ref_div}
	D'_i\,=\,\mathbb{V}(z_i) ~ (\text{respectively}, D_i\,=\,\mathbb{V}(w_i) ) \text{ for } 1 \,\leq\, i \,\leq\,
	n
\end{equation}
 (see \cite[Example 5.2.5]{Cox}).

We have \[U_n \,\cong\, U_{n-1} \times (\mathbb{C}^2 \setminus {0}),\ \
(z_1,w_1, \cdots, z_n, w_n) \,\longmapsto\, ((z_1,w_1, \cdots, z_{n-1}, w_{n-1}) \times (z_n, w_n))\] and 
\[G_n \,\cong\, G_{n-1} \times \mathbb{C}^*,\ \ (t_1,\, \cdots,\, t_n) \,\longmapsto\, ((t_1, \,\cdots,\, t_{n-1}),
\, t_n),\]
where the last factor \(t_n\) acts trivially on \(U_{n-1}\). Thus \(X_{n-1}=U_{n-1}/\!\!/G_{n-1}\) is the Bott tower associated to the Bott numbers $\{c_{i, j}\}_{\{1\leq i < j \leq n-1 \}}$. This also induces the map 
\begin{equation*}
X_n \,\longrightarrow\, X_{n-1},\ \ [z_1:w_1: \ldots,:z_n:w_n] \,\longmapsto\, [z_1:w_1: \ldots: z_{n-1}:w_{n-1}].
\end{equation*}
In general, for each \(1 \,\leq\, i \,\leq\, n\), there is a map 
\begin{equation*}
X_i \,\longrightarrow\, X_{i-1},\ \ [z_1:w_1: \ldots: z_i:w_i] \,\longmapsto\,
[z_1:w_1: \ldots: z_{i-1}:w_{i-1}]
\end{equation*}
together with a section given by
\begin{equation*}\label{qt_section}
X_{i-1}\,\longrightarrow\, X_i,\ \ [z_1:w_1: \ldots: z_{i-1}:w_{i-1}] \,\longmapsto\,
[z_1:w_1: \ldots: z_{i-1}:w_{i-1}:0:1]\, .
\end{equation*}

\subsection{Basic set-up }

Fix a point $x\,\in \,X_n$. Now we describe a special class of 
subvarieties
$X_n^{(j)}$ of $X_n$ for $1 \,\le\, j \,\le\, n$ equipped with rational curves
$\Gamma_n^{(j)} \,\subset\, X_n^{(j)}$. We emphasize that
these subvarieties and rational curves \textit{depend on the given point
$x$}. However, for convenience, we omit indicating this in the notation. 

Set $X_i^{(1)} \,:=\, X_i$ for every $1 \,\le\, i \,\le\, n$. For every $2 \,\le\, i \,\le\, n$, 
let $$\pi_i\,:\, X_i \,\longrightarrow\, X_1$$ be the composition of maps in
\eqref{bn}. Define $X_i^{(2)} \,:=\,
\pi_i^{-1}(\pi_n(x))$. Note that $x \,\in\, X_n^{(2)}$. 

Then $X_n^{(2)}\,\longrightarrow\, X_{n-1}^{(2)}\,\longrightarrow\, \cdots
\,\longrightarrow\, X_2^{(2)}$ is a Bott tower (see Proposition
\ref{vertical-tower} below). 
For every $3 \,\le\, i \,\le\, n$, 
let $\pi_{2,i}\,:\, X_i^{(2)} \,\longrightarrow\, X_2^{(2)}$ be the composition of
these maps. Define $X_i^{(3)} \,:=\,\pi_{2,i}^{-1}(\pi_{2,n}(x))$. 

Proceeding this way, we define $X_i^{(j)}$ for every $1\,\le\, j \,\le \,i \,\le\,
n$. Note that $x \in X_n^{(j)}$ for all $1 \, \le \, j \, \le \, n$. 
Further, $X_i^{(i)} \,= \,\mathbb{P}^1$ for each $1 \,\le\, i \,\le\, n$.
See Figure \ref{diag1} below.

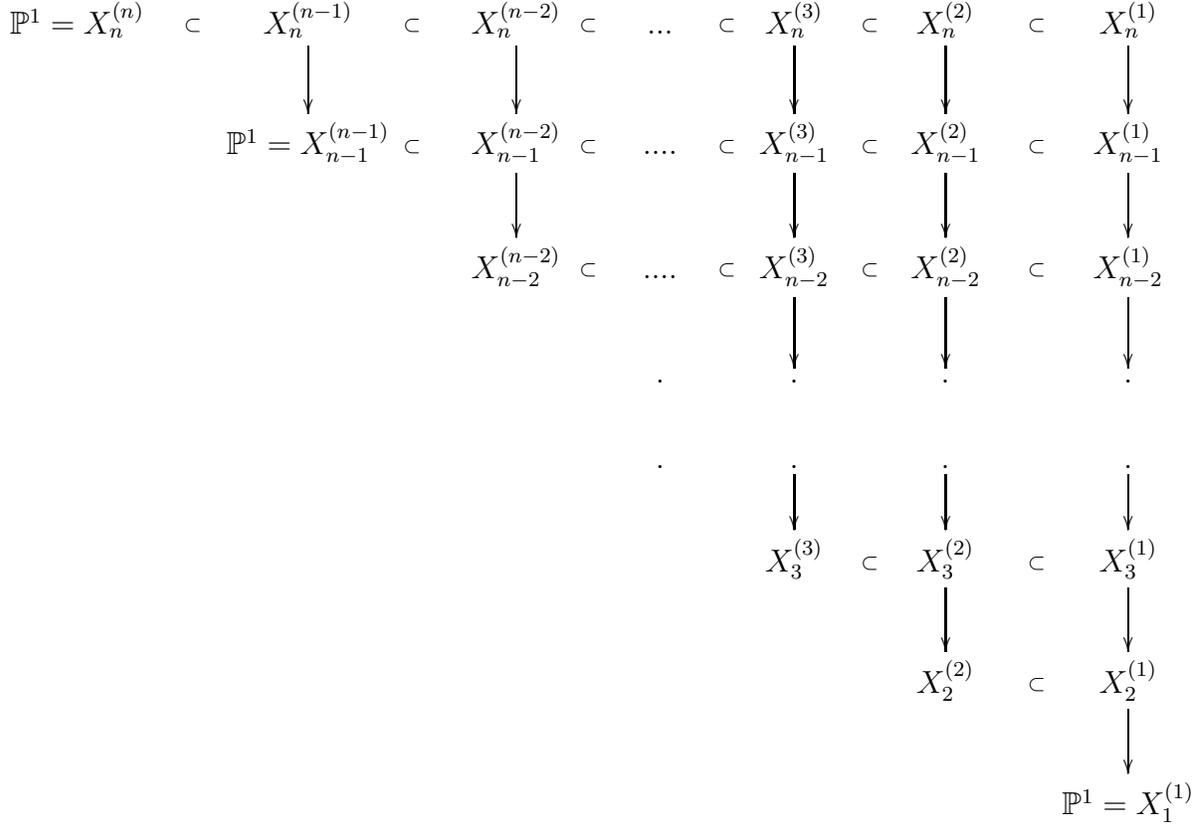
\begin{figure}[h]
$$
\xymatrix{
\mathbb{P}^1 = X_n^{(n)}\ar@{}[r]|{\subset} & X_n^{(n-1)}\ar@{}[r]|{\subset}\ar[d]
& X_n^{(n-2)}\ar@{}[r]|{\subset}\ar[d] & ...\ar@{}[r]|{\subset} & X_n^{(3)}\ar@{}[r]|{\subset}\ar[d] & X_n^{(2)}\ar@{}[r]|{\subset}\ar[d] & X_n^{(1)}\ar[d]\\
 &\mathbb{P}^1 = X_{n-1}^{(n-1)}\ar@{}[r]|{\subset} & X_{n-1}^{(n-2)}\ar@{}[r]|{\subset}\ar[d]  & ....\ar@{}[r]|{\subset}  & X_{n-1}^{(3)}\ar@{}[r]|{\subset}\ar[d]  &  X_{n-1}^{(2)}\ar@{}[r]|{\subset}\ar[d]  & X_{n-1}^{(1)}\ar[d]\\
& & X_{n-2}^{(n-2)}\ar@{}[r]|{\subset}  &
....\ar@{}[r]|{\subset}  & X_{n-2}^{(3)}\ar@{}[r]|{\subset}\ar[d]  &
X_{n-2}^{(2)}\ar@{}[r]|{\subset}\ar[d]  &
X_{n-2}^{(1)}\ar[d]\\
&&&.&.&.&.&\\
&&&.&.\ar[d]&.\ar[d]&.\ar[d]& \\
&&&&X_3^{(3)} \ar@{}[r]|{\subset} &X_3^{(2)} \ar@{}[r]|{\subset}\ar[d]&X_3^{(1)} \ar[d]\\
&&&&&X_2^{(2)} \ar@{}[r]|{\subset}&X_2^{(1)}\ar[d]\\
&&&&&& \mathbb{P}^1 = X_1^{(1)}}
$$	
\caption{Construction of $X_i^{(j)}, 1 \, \le \, j  \, \le i \le n$}
	\label{diag1}
	\end{figure}

\begin{proposition}\label{vertical-tower}
Each vertical tower in Figure \ref{diag1} is a Bott tower with positive invariants.
\end{proposition}

\begin{proof}
Fix a point \(x\,=\,[z_1^0:w_1^0: \ldots :z_n^0: w_n^0]\, \in\, X_n\). Then for \(j \leq i\),
\begin{equation*}
\begin{split}
X_i^{(j)}\,=\,
\{[z_1^0:w_1^0: \ldots :z_{j-1}^0:w_{j-1}^0:z_j:w_j:z_{j+1}:&w_{j+1}:\ldots
:z_i:w_i]\,\,\mid\\
& (z_l,w_l) \in \C^2 \setminus 0 \text{ for } j \leq l \leq i \} \subset X_i.
\end{split}
\end{equation*}
This can be identified with a Bott tower of dimension \({i-j+1}\) with Bott numbers
\(\{c_{k,l}\}_{\{j\leq k<l \leq i\}}\) via the map \[[z_1^0:w_1^0: \ldots :z_{j-1}^0:w_{j-1}^0:z_j:w_j:\ldots :z_i:w_i]
\,\longmapsto\, [z_j:w_j:\ldots :z_i:w_i].\] Similarly \(X_{i-1}^{(j)}\) (provided \(j \leq i-1\)) can be identified with a Bott tower of dimension \({i-j}\) with Bott numbers \(\{c_{k,l}\}_{\{j\leq k<l \leq i-1\}}\). Also note that the map \(X_i^{(j)} \rightarrow X_{i-1}^{(j)}\) is defined by
\begin{equation*}
\begin{split}
[z_1^0:w_1^0: \ldots :z_{j-1}^0:w_{j-1}^0:&z_j:w_j:\ldots :z_i:w_i]\\
& \longmapsto\, [z_1^0:w_1^0: \ldots :z_{j-1}^0:w_{j-1}^0:z_j:w_j:\ldots :z_{i-1}:w_{i-1}].
\end{split}
\end{equation*}
Thus each vertical tower in Figure \ref{diag1} of the form
$$X_n^{(j)} \,\longrightarrow\, X^{(j)}_{n-1} \,\longrightarrow\, \ldots \,\longrightarrow \,X^{(j)}_{j+1}
\,\longrightarrow\, X^{(j)}_j$$ is a Bott tower with Bott numbers \(\{c_{k,l}\}_{\{j\leq k<l \leq n\}}\).
Since all the Bott numbers were assumed to be positive, this completes the proof.
\end{proof}

\begin{proposition}\label{fibre}
Let \(x \,\in\, X_n\) be a point. Then $X_n^{(2)} \,\equiv\, D_1$.
\end{proposition}

\begin{proof}
Let \(x=[z_1^0:w_1^0: \ldots : z_n^0: w_n^0] \,\in\, X_n\) be as before. Then \[X^{(2)}_n
\,=\,\{[z_1^0:w_1^0: z_2:w_2 : \ldots: z_n:w_n]~ |~ (z_i,w_i) \in \C^2 \setminus 0\ \text{ for } 2 \leq i \leq n \}.\]
Note that if \(z_1^0\,=\,0\), then \(X^{(2)}_n\,=\,D'_1 \sim_{\text{lin}} D_1\) (by \eqref{linequiv}),
and \(X^{(2)}_n\,=\,D_1\) when \(w_1^0\,=\,0\). So we can assume both \(z_1^0,w_1^0\) are
non-zero. Since \(X^{(2)}_n\) is a divisor in \(X_n\), we can write 
\begin{equation}\label{numequiv}
X^{(2)}_n \,\equiv\, b_1 D_1 + \ldots + b_nD_n,\ \text{ with } b_1,\, \cdots,\, b_n \,\in\, \Z.
\end{equation}
Recall that the \((n-1)\)-dimensional cones in the fan correspond to invariant curves in the toric
variety; let \(V(\tau)\) denote the invariant curve corresponding to the \((n-1)\)-dimensional
cone \(\tau\) in the fan. Now consider the curves \(C_i\,=\,V(\tau_i)\) in \(X_n\),
where \(\tau_i\,=\,\text{Cone}(v_1,\, \cdots,\, \widehat{v}_i, \,\cdots,\, v_n)\) are the \((n-1)\)-dimensional cones
in the fan \(\Delta_n\). Since \(D_i \cdot C_j\,=\,\delta_{ij}\) (see
\cite[Corollary 6.4.3, Proposition 6.4.4]{Cox}),\, \eqref{numequiv} gives that \(b_i\,=\,X^{(2)}_n \cdot C_i\). 
	
Let \(i\,>\,1\). Then \(C_i\,=\,D'_1 \cap \left(  \cap_{j\neq 1, i} D'_{j}\right) \).
This implies that \(X^{(2)}_n  \cap C_i\,=\, \emptyset\), i.e., \(b_i\,=\,0\).
Now \(X^{(2)}_n \cap C_1=\{[z_1^0:w_1^0:0:1: \ldots : 0:1]\}\). So \(X^{(2)}_n \cdot C_1=1\),
i.e., \(b_1=1\). Thus  $X_n^{(2)} \,\equiv\, D_1$.
\end{proof}

Let $1 \,\le\, i \,\le\, n$. Let $X_n^{(i)} \longrightarrow X^{(i)}_{n-1}\longrightarrow \ldots
\longrightarrow X^{(i)}_{i+1}\longrightarrow X^{(i)}_i$ be a vertical Bott tower 
in Figure \ref{diag1}. Then by the construction of projective bundles,
there is a section map 
$X_{j-1}^{(i)} \,\longrightarrow\, X_{j}^{(i)}$ for every $i+1 \,\le\, j \,\le\, n$. 

For each $1\,\le\, i \,\le\, n$, let $\sigma_i\,:\, X_i^{(i)} \,\longrightarrow\, X_n^{(i)}$ be the composition
of section maps. Define
\begin{equation}\label{gni}
\Gamma_n^{(i)} \,:=\,\sigma_i( X_i^{(i)})\, .
\end{equation}
We have $\Gamma_n^{(i)} \subset X_n^{(i)}$ for each $i$ and
$\Gamma_n^{(n)} \,=\, X_n^{(n)}$. We denote $\Gamma_n^{(1)}$ also by
$\Gamma_n$. See Figure \ref{diag2} below.

\begin{figure}[h]
	$$
	\xymatrix{
		X_n^{(n)}\ar@{}[r]|{\subset}  &  X_n^{(n-1)}\ar@{}[r]|{\subset}
		& ... \ar@{}[r]|{\subset}  &
		X_n^{(3)}\ar@{}[r]|{\subset}  &  X_n^{(2)}\ar@{}[r]|{\subset} &
		X_n^{(1)}\\
		\Gamma_n^{(n)}\ar@{}[u]|{||}  &  \Gamma_n^{(n-1)}\ar@{}[u]|{\cup}
		& ...  &
		\Gamma_n^{(3)}\ar@{}[u]|{\cup}  &  \Gamma_n^{(2)}\ar@{}[u]|{\cup} &
		\Gamma_n = \Gamma_n^{(1)}\ar@{}[u]|{\cup}\\
		\sigma_n( X_n^{(n)})\ar@{}[u]|{||}  &  \sigma_{n-1}( X_{n-1}^{(n-1)})\ar@{}[u]|{||}
		& ...  &
		\sigma_3( X_3^{(3)})\ar@{}[u]|{||} &  \sigma_2( X_2^{(2)})\ar@{}[u]|{||} &
		\sigma_1( X_1^{(1)})\ar@{}[u]|{||}}
	$$	
	\caption{Construction of $\Gamma_n^{(i)}, 1 \le i \le n$}
	\label{diag2}
\end{figure}

\begin{proposition}\label{dual-basis}
The curves $\Gamma_n,\, \Gamma_n^{(2)},\,\cdots,\,\Gamma_n^{(n)}$ defined in \eqref{gni} span 
$\overline{\rm{NE}}(X_n)$, and they are dual to $D_1,\, \cdots ,\, D_n$. 
\end{proposition}

\begin{proof}
Fix a point \(x\,=\,[z_1^0:w_1^0: \ldots: z_n^0: w_n^0]\, \in \,X_n\). Then for \(1 \,\leq \,i \,\leq\, n\),
\begin{equation}\label{curve}
	\Gamma_n^{(i)}\,=\,\{[z_1^0:w_1^0: \ldots :z_{i-1}^0:w_{i-1}^0:z_i:w_i:0:1:\ldots :0:1]\,\mid\, (z_i,w_i)\, \in\,
\C^2 \setminus 0 \}\, \subset\, X_n^{(i)}.
\end{equation}
Now \(D_1 \cap \Gamma_n^{(1)}\,=\,\{[1:0:0:1:\ldots :0:1]\}\), and hence 
\begin{equation*}
D_1 \cdot \Gamma_n^{(1)}\,=\,1.
\end{equation*}
For \(1\,<\,i \,\leq\, n\), we have 
\begin{equation*}
D_1 \cap \Gamma_n^{(i)} \,=\,
\begin{cases}
\emptyset, & {\rm if}\ w_1^0 \neq 0, \\
\Gamma_n^{(i)}, & {\rm if}\ w_1^0 = 0.
\end{cases}
\end{equation*} 
	 
Thus \( D_1 \cdot  \Gamma_n^{(i)}\,=\,0\) when \(w_1^0 \,\neq\, 0\).  But
\(D_1 \sim_{\text{lin}} D'_1\) by \eqref{linequiv}, and \(D'_1 \cap  \Gamma_n^{(i)}\,=\, \emptyset \)
when \(w_1^0\,=\,0\) because
\((z_1^0, w_1^0) \,\in\, \C^2 \setminus 0\). Therefore, \(D_1 \cdot  \Gamma_n^{(i)}\,=\,0\)
when \(1 \,<\,i \,\leq\, n\). Thus for \(1 \leq i \leq n\), we have
 \begin{equation*}
 D_1 \cdot  \Gamma_n^{(i)}\,=\,\delta_{i1}.
 \end{equation*}
Fix \(1\,<\,j \,\leq\, n\), and assume that 
 \begin{equation}\label{moriassump}
 D_k \cdot  \Gamma_n^{(i)}\,=\,\delta_{ik}
 \end{equation}
whenever $1 \,\leq\, i \,\leq\, n$ and $1 \,\leq\, k\, <\,j$. 

We will show that \( D_j \cdot  \Gamma_n^{(i)}\,=\,\delta_{ij}\) for \(1 \,\leq \,i \,\leq\, n\). 
 First note that $$D_j \cap \Gamma_n^{(j)}\,=\,\{[z_1^0:w_1^0: \ldots
 :z_{j-1}^0:w_{j-1}^0:1:0: 0:1:\ldots :0:1] \}$$ 
and hence \( D_j \cdot  \Gamma_n^{(j)}=1\). For \(1 \leq i \leq n, i \neq j\), we have 
 \begin{equation*}
 	D_j \cap  \Gamma_n^{(i)} =
 	\begin{cases}
 		\emptyset, &
 		{\rm if}\ w_j^0 \neq 0, \\
 		\Gamma_n^{(i)},  & {\rm if}\ w_j^0 = 0.
 	\end{cases}
 \end{equation*} 
Thus for \(w_j^0 \neq 0\), we have that \( D_j \cdot
\Gamma_n^{(i)}=0\). On the other hand, if \(w_j^0 =0\), then $z_j^0 \neq 0$ and hence 
\begin{equation}\label{mori1}
D'_j \cdot  \Gamma_n^{(i)}=0. 
\end{equation}
Again from \eqref{linequiv}, we have that 	
\begin{equation}\label{mori2}
D_j \sim_{\text{lin}} D'_{j}+c_{1, j}D_{1}+\ldots+c_{j-1, j}D_{j-1}.
\end{equation}
Note that \(w_j^0 =0\) is possible only if \(j<i\), and then by \eqref{moriassump}, we have that 
\begin{equation}\label{mori3}
D_k \cdot  \Gamma_n^{(i)}\,=\,0 
\end{equation}
holds for $1 \,\leq\, k\, <\,j$.

Then from \eqref{mori1}, \eqref{mori2} and \eqref{mori3}, it follows that
\(D_j \cdot  \Gamma_n^{(i)}\,=\,0. \) Hence \( D_j \cdot 
\Gamma_n^{(i)}\,=\,\delta_{ij}\) holds also for \(1 \,\leq\, i \,\leq\, n\). 

So $\Gamma_n,\, \Gamma_n^{(2)},\,\cdots,\,\Gamma_n^{(n)}$ are
dual to $D_1, \,\cdots ,\, D_n$. This implies that they span
$\overline{\text{NE}}(X_n)$, because $D_1,\,\cdots,\,D_n$ span the nef cone
of $X_n$ and $\overline{\text{NE}}(X_n)$ is dual to the nef cone of
$X_n$. 
\end{proof}

\begin{remark}\label{alt-char} As an immediate consequence of the descriptions in \eqref{ref_div} and \eqref{curve}, we have
the following alternative characterization of the curve $\Gamma_n$: $$\Gamma_n = D_2' \cap \ldots \cap D_n'.$$ 
More generally, for a point \(x\,=\,[z_1^0:w_1^0: \ldots: z_n^0: w_n^0]\, \in \,X_n\), we have
$$x \in \Gamma_n^{(i)} \text{~if and only if~}  x \in D_{i+1}' \cap \ldots \cap D_n'.$$
\end{remark}

From \eqref{curve}, we see that  \( x \in \Gamma_n^{(i)}\) implies that \( x\) is also in \(\Gamma_n^{(i+1)}\). On the other hand, let \(x \in X_n\), and let \(C\) be an irreducible and reduced curve containing $x$. Assume
that $x \,\notin\,\Gamma_n$. Then $C \,\ne \,\Gamma_n$. More generally, we have the following:
 
\begin{lemma}\label{containment}
For any \(2 \,\leq\, i \,\leq\, n\), if $x \,\in\, \Gamma_n^{(i)}\setminus \Gamma_n^{(i-1)}$,
then $C \,\not\subset\, D_i'$.
\end{lemma}	

\begin{proof}
Write \(x\,=\,[z_1^0:w_1^0 : \ldots : z_n^0:w_n^0] \,\in\, X_n\).
The condition that $x\, \in\, \Gamma_n^{(i)} $ implies that \(z_j^0\,=\,0\) for all
\(j \,\geq\, i+1\). Similarly $x \,\notin\, \Gamma_n^{(i-1)}$ implies that
\(z_j^0 \,\neq\, 0\) for some \(j\,\geq\, i\). Thus 
\(z_i^0 \,\neq\, 0\), i.e., \(x \,\notin\, D'_i\). Since \(x \,\in\,
C\), it follows that $C
\,\not\subset\, D_i'$.
\end{proof}	

\begin{lemma}\label{resfibre}
Let \(L \,\equiv\, (a_1,\, \cdots,\, a_n) \,\in\, {\rm Pic}(X_n)\). Then
$L|_{X_n^{(i)}} \,\equiv\, (a_i,\,\cdots,\, a_n)$ whenever \(2 \,\leq\, i \,\leq\, n\). 
\end{lemma}

\begin{proof}
Let \(x\,=\,[z_1^0:w_1^0 : \ldots : z_n^0:w_n^0] \,\in\, X_n\) be as before. Note that $$X_n^{(i)}\,=\,\left\{[z_1^0:w_1^0 : \cdots :
z_{i-1}^0:w_{i-1}^0:z_i:w_i: \cdots :z_n:w_n] \,\mid\, (z_l,w_l)
\,\in\, \C^2 \setminus 0\ \text{ for } i \leq l \leq n
\right\}$$ 
is isomorphic to $$\left\{[z_1':w_1': \cdots : z_{n-i+1}': w_{n-i+1}'] \,\mid\,
(z'_l,\,w'_l) \,\in \, \C^2 \setminus 0\ \text{ for } 1 \,\leq\, l\, \leq\,
n-i+1 \right\},$$ 
which is a Bott tower of dimension \(n-i+1\), where we have 
identified \(z_j'\) (respectively, \(w'_j\)) with \(z_{i-1+j}\) (respectively, $w_{i-1+j}$).
Note that \[\text{Pic}(X_n^{(i)})\,=\,\Z D^{(i)}_1\oplus \cdots\oplus \Z D^{(i)}_{n-i+1}\, ,\]
where \(D^{(i)}_j\,=\,\mathbb{V}(w'_{j})\) for $j\,=\,1,\, \cdots,\, n-i+1.$
	
Let \(i \,\leq\, j \,\leq\, n\), and consider
\begin{align*}
	D_j \cap X_n^{(i)} &=\{[z_1^0:w_1^0 : \ldots : z_{i-1}^0:w_{i-1}^0:z_i:w_i: \ldots :z_n:w_n] \in X_n \mid w_j=0\}\\
	&= \mathbb{V}(w'_{j-i+1}) \subseteq X_n^{(i)}\\
	&= D^{(i)}_{j-i+1}.
\end{align*}
	
Now using induction on \(j\), we show that \(D_j|_{X_n^{(i)}}=0\) for \(1 \leq j <i\). Note that 
	\begin{equation}\label{rest}
		D_j \cap  X_n^{(i)}=
		\begin{cases}
			\emptyset, &
			{\rm if}\ w_j^0 \neq 0, \\
			X_n^{(i)},  & {\rm if}\ w_j^0 = 0.
		\end{cases}
	\end{equation} 
	 First suppose that \(j=1\). From \eqref{rest}, if  $w_1^0
         \neq 0$ then \(D_1|_{X_n^{(i)}}=0\). So assume \(w_1^0 =
         0\). Then from \eqref{linequiv}, 
we have \(D_1 \sim_{\text{lin}} D'_1\). Since \(w_1^0 = 0\), we have
\(D'_{1}\cap  X_n^{(i)} = \emptyset\), which shows that
\(D_1|_{X_n^{(i)}}=0\). 
Now assume that \(D_l|_{X_n^{(i)}}=0\) for all \(l <j\). Again from
\eqref{rest}, if  $w_j^0 \neq 0$ then \(D_j|_{X_n^{(i)}}=0\).

So let \(w_j^0 = 0\). Then from \eqref{linequiv}, 
	\begin{equation}\label{linequiv2}
		D_j \sim_{\text{lin}} D'_{j}+c_{1, i}D_{1}+\ldots+c_{j-1, j}D_{j-1}.
	\end{equation}
	
Clearly, \( D'_{j}\cap  X_n^{(i)} = \emptyset\) as \(w_j^0 = 0\). Now
from \eqref{linequiv2} together with the induction hypothesis, it
follows that 
\(D_j|_{X_n^{(i)}}=0\). This completes the proof.
\end{proof}

\section{Seshadri constants}\label{sc}

In this section we prove our main theorem which determines the
Seshadri constants of nef line bundles on Bott towers. We follow
the notation developed in Section \ref{bott} about Bott towers. 

Let $n$ be a positive integer. We
consider Bott towers $$X_n \,\longrightarrow\,
 X_{n-1}\,\longrightarrow\, \ldots \,\longrightarrow\, X_2 \,\longrightarrow\, X_1
\,\longrightarrow\, X_0$$ of height $n$. Then $X_1 \,=\, \mathbb{P}^1$. If $L
\,=\,\mathcal{O}_{\mathbb{P}^1}(a)$ is a nef
line bundle on $\mathbb{P}^1$, then by convention, the Seshadri constants of $L$ are
given by
$\varepsilon(\mathbb{P}^1,\mathcal{O}_{\mathbb{P}^1}(a),x)\,=\,a$
for every $x \,\in\, \mathbb{P}^1$. 

\begin{theorem}\label{main}
Let $n$ be a positive integer. 
Let $$X_n \,\longrightarrow\,
 X_{n-1}\,\longrightarrow\, \ldots \,\longrightarrow\, X_2 \,\longrightarrow\, X_1
\,\longrightarrow\, X_0$$ be a Bott
tower with positive Bott numbers. Take any $x \,\in\, X_n$.  Let $X_n^{(2)}$ be the subvariety of $X_n$ defined in Figure
\ref{diag1} and let $\Gamma_n \subset X_n$ be the rational curve defined in Figure
\ref{diag2}. 
Let $L \,\equiv\, a_1 D_1 + \ldots + a_n D_n$ be a
nef line bundle on $X_n$. Then the Seshadri constants of $L$ are given by the
following. 
\begin{equation*}
    \varepsilon(X_n,L,x) =
    \begin{cases}
      {\rm min}\{a_1,\varepsilon(X_n^{(2)},L|_{X_n^{(2)}},x)\}, &
      {\rm if}\ x\in \Gamma_n, \\
      \varepsilon(X_n^{(2)},L|_{X_n^{(2)}},x),  & {\rm if}\ x\notin \Gamma_n.
    \end{cases}
  \end{equation*} 
\end{theorem}

\begin{remark}\label{alt-main}
By Remark \ref{alt-char},  the conclusion of Theorem \ref{main} can re-phrased as
\begin{equation*}
    \varepsilon(X_n,L,x) =
    \begin{cases}
      {\rm min}\{a_1,\varepsilon(X_n^{(2)},L|_{X_n^{(2)}},x)\}, &
      {\rm if}\ x\in D_2' \cap \ldots \cap D_n', \\
      \varepsilon(X_n^{(2)},L|_{X_n^{(2)}},x),  & {\rm if}\ x\notin D_2' \cap \ldots \cap D_n'.
    \end{cases}
  \end{equation*} 
\end{remark}
\begin{remark}\label{n=1}
In Section \ref{bott} we defined subvarieties $X_n^{(i)}$ of $X_n$
for every $i \,\le\, n$; see Figure \ref{diag1}. So $X_n^{(i)}$ is not
defined when $i\, >\, n$. When $n\,=\,1$, we note that $X_1\, =\, \Gamma_1$. Therefore,
Theorem \ref{main} is to be interpreted as
follows when $n\,=\,1$: Since $X_1^{(2)}$ is not defined,
the theorem asserts that 
$\varepsilon(\mathbb{P}^1,\mathcal{O}_{\mathbb{P}^1}(a),x)\,=\,a$ for all
$x\, \in \,\mathbb{P}^1$. This
holds by the convention on Seshadri constants for nef line bundles
on $\mathbb{P}^1$. 
\end{remark}

We will prove Theorem \ref{main} in this section. 
We start with some propositions. 

\begin{proposition}\label{gen-ineq} With the notation as in Theorem
  \ref{main},
$$\varepsilon(L,x) \,\ge\,  {\rm
  min}\{a_1,\,\varepsilon(X_n^{(2)},\,L|_{X_n^{(2)}},x)\}$$
for all $x \,\in\, X_n$.
\end{proposition}

\begin{proof}
Let $C \,\subset\, X_n$ be an irreducible and reduced curve such that $m\,:=\,
\text{mult}_x(C) \,>\, 0$. Write $C \,=\, p_1 \Gamma_n +
p_2\Gamma_n^{(2)}+\ldots + p_n\Gamma_n^{(n)}$, for some non-negative
integers $p_1,\, \ldots ,\, p_n$. 

First suppose that $C \,\not\subset\, X_n^{(2)}$. Then, using Proposition
\ref{fibre},  we have 
$$C \cdot X_n^{(2)} \,=\, C \cdot D_1 \,\ge\, m (\text{mult}_x(X_n^{(2)})) \,\ge\, m,$$ by B\'ezout's
theorem. Proposition \ref{dual-basis} implies  $C \cdot D_1 \,=\,
p_1$. So $p_1 \,\ge\, m$. Thus 
$$\frac{L\cdot C}{m} \,=\, \frac{a_1p_1+\ldots+a_np_n}{m} \,\ge\, a_1.$$ 

Next suppose that $C\,\subset\, X_n^{(2)}$. Then from the definition of 
$\varepsilon(L|_{X_n^{(2)}},x)$ it follows that 
$$\frac{L\cdot C}{m} \,=\, \frac{L|_{X_n^{(2)}} \cdot C}{m} \,\ge\, \varepsilon(L|_{X_n^{(2)}},\,x).$$

Consequently, $\frac{L\cdot C}{m} \,\ge\,
\text{min}\{a_1,\varepsilon(L|_{X_n^{(2)}},x)\}$ for all irreducible
and reduced curves $C \,\subset\, X_n$. The proposition follows. 
\end{proof}

\begin{proposition} \label{case1}
For $n\, >\, 0$, let $$X_n \,\longrightarrow\,
 X_{n-1}\,\longrightarrow\, \ldots \,\longrightarrow\, X_2 \,\longrightarrow\, X_1
\,\longrightarrow\, X_0$$ be a Bott
tower with positive Bott numbers. Let $L \,\equiv\, a_1 D_1 + \ldots +
a_n D_n$ be a nef line bundle
on $X_n$. Take any $x \,\in\, \Gamma_n\, :=\, \Gamma_n^{(1)}$ $($see \eqref{gni}$)$. Then
$$\varepsilon(L,x) \,=\, {\rm
  min}\{a_1,\varepsilon(X_n^{(2)},\,L|_{X_n^{(2)}},x)\}.$$ 
\end{proposition}

\begin{proof}
We know that $\Gamma_n$ is a smooth rational curve. Since $x \,\in\,
\Gamma_n$, we have $\varepsilon(L,x) \le L\cdot \Gamma_n \,=\, a_1$. 

On the other hand, we always have 
$$\varepsilon(L,x) \,=\, \inf\limits_{\substack{x \in C \subset X_n}} \frac{L\cdot
C}{{\rm mult}_{x}C} \,\le \,\inf\limits_{\substack{x \in C\subset X_n^{(2)}}} \frac{L\cdot
C}{{\rm mult}_{x}C} \,=\, \varepsilon(L|_{X_n^{(2)}},x).$$

So  $\varepsilon(L,x) \,\le\,  {\rm
  min}\{a_1,\varepsilon(X_n^{(2)},\,L|_{X_n^{(2)}},x)\}$ and the
proposition follows from Proposition \ref{gen-ineq}. 
\end{proof}

\begin{lemma}\label{lemma} With the notation as in Theorem
  \ref{main}, suppose that $$\frac{L\cdot C}{{\rm mult}_{x}C} \,\ge\,
\varepsilon(X_n^{(2)},\,L|_{X_n^{(2)}},x)$$ 
for all irreducible and reduced curves $C \,\not\subset\, X_n^{(2)}$. Then 
$\varepsilon(L,x) \,=\, \varepsilon(X_n^{(2)},\,L|_{X_n^{(2)}},x)$.  
\end{lemma}

\begin{proof}
For any irreducible and reduced curve $C \,\subset\, X_n^{(2)}$, we have 
\begin{equation}\label{c1}
\frac{L\cdot C}{{\rm mult}_{x}C} \,\ge\,
\varepsilon(X_n^{(2)},\,L|_{X_n^{(2)}},x)
\end{equation}
by the definition of Seshadri
constants. So by the given hypothesis, \eqref{c1} holds for all
curves in $X_n$. So $\varepsilon(L,x) \,\ge\,
\varepsilon(X_n^{(2)},\,L|_{X_n^{(2)}},x)$. Since the opposite
inequality $\varepsilon(L,x) \,\leq\,
\varepsilon(X_n^{(2)},\,L|_{X_n^{(2)}},x)$ holds, the lemma is proved. 
\end{proof}

Next we consider the second case of Theorem \ref{main}. Note that if
$x \,\notin\, \Gamma_n$, then $n\,\ge\, 2$. 

\begin{proposition}\label{case2}
Let $n \ge 2$ be an integer. 
Let $$X_n \,\longrightarrow\,
 X_{n-1}\,\longrightarrow\, \ldots \,\longrightarrow\, X_2 \,\longrightarrow\, X_1
\,\longrightarrow\, X_0$$ be a Bott
tower with positive Bott numbers. Let $L\, \equiv\, a_1 D_1 + \ldots +
a_n D_n$ be a nef line bundle
on $X_n$. Take $x \,\in\, X_n$ such that $x \,\notin\, \Gamma_n$. Then 
$\varepsilon(X_n,\,L,x) \,=\, \varepsilon(X_n^{(2)},\,L|_{X_n^{(2)}},x)$. 
\end{proposition}

\begin{proof}
We prove this by using induction on $n$. First set $n=2$. This implies that 
$X_2^{(2)} \,=\, \mathbb{P}^1$ and $L|_{X_2^{(2)}}\, =\, (a_2)$ by Lemma \ref{resfibre}. 
Hence $\varepsilon(X_n^{(2)},\,L|_{X_n^{(2)}},x)\,=\,a_2$. 
So we need to prove that 
$\varepsilon(X_2,\,L,x) \,=\, a_2$, under the assumption that $x \notin
\Gamma_2$.

Let $C \,=\, p_1\Gamma_2+p_2\Gamma_2^{(2)} \,\subset\, X_2$ be an irreducible
and reduced curve for some non-negative
integers $p_1, p_2$ such that $m\,:=\,{\rm mult}_x(C) \,>\, 0$. 
Note that $x \,\in\, \Gamma_2^{(2)} \,=\, X_2^{(2)}$. Hence 
$x \,\in\,\Gamma_2^{(2)} \setminus \Gamma_2$ and we have  
$C \,\not\subset\, D_2'$ by Lemma \ref{containment}. 
By \eqref{linequiv}, $$D_2' \,\sim_{\text{lin}}\, D_2-c_{1,2}D_1\, ,$$ where
$c_{1,2}$ is a positive integer. Hence  
$$0 \,\le\, C\cdot D_2' \,=\, C\cdot(D_2-c_{1,2}D_1) \,=\, p_2-c_{1,2}p_1\, .$$ So $p_2
\,\ge\, p_1$. 

Now suppose that $C \,\ne \, \Gamma_2^{(2)} \,=\, X_2^{(2)}$. Then 
B\'ezout's theorem and Proposition \ref{fibre} together give that
$m \le C \cdot X_2^{(2)} = C \cdot D_1 = p_1.$
Hence $$\frac{L\cdot C}{m} \,=\, \frac{a_1p_1+a_2p_2}{m} \,\ge\, a_2.$$ 

On the other hand, since $x \,\in\, \Gamma_2^{(2)}$, it contributes to the
Seshadri constant of $L$ at $x$. So $\varepsilon(X_2,L,x) \,\le\,
L \cdot \Gamma_2^{(2)} \,=\, a_2$. Since 
$\frac{L\cdot C}{m} \,\ge\, a_2$ for every curve $C \,\ne\, \Gamma_2^{(2)}$, it
follows that $\varepsilon(X_2,L,x)\,=\,a_2$. 

Now assume that
\begin{itemize}
\item  $n \,\ge\, 3$, and

\item the proposition holds for all Bott towers of height at most $n-1$.  
\end{itemize}

Note that $L|_{X_n^{(2)}} \,\equiv\, (a_2,\,\cdots,\,a_n)$ by Lemma \ref{resfibre}. 
By the induction hypothesis and Proposition \ref{case1}, we know that
\begin{equation*}
    \varepsilon(X_n^{(2)},\,L|_{X_n^{(2)}},x) \,=\,
    \begin{cases}
      {\rm min}\{a_2,\varepsilon(X_n^{(3)},\,L|_{X_n^{(3)}},x)\}, &
      {\rm if}\ x\,\in \,\Gamma_n^{(2)}, \\
      \varepsilon(X_n^{(3)},L|_{X_n^{(3)}},\, x),  & {\rm if}\ x\,\notin\, \Gamma_n^{(2)}.
    \end{cases}
  \end{equation*} 

\textbf{Case 1}: $x\,\in\, \Gamma_n^{(2)}$. 

In this case, $\varepsilon(X_n^{(2)},\,L|_{X_n^{(2)}},x) \,=\,  {\rm
  min}\{a_2,\,\varepsilon(X_n^{(3)}\,,L|_{X_n^{(3)}},\,x)\} \,\le\, a_2$.

\begin{lemma} Let $C \,\subset\, X_n$ be an irreducible and reduced
curve such that $C \,\not\subset\, X_n^{(2)}$ and $m\,:=\, \text{mult}_x(C)\, >\,
0$. Then $\frac{L\cdot C}{m}\, \ge\, a_2$. 
\end{lemma}
\begin{proof}
 Write $C \,=\, p_1 \Gamma_n +
p_2\Gamma_n^{(2)}+\ldots + p_n\Gamma_n^{(n)}$ for some non-negative
integers $p_1,\, \cdots ,\, p_n$. 

Since $x \,\in\, \Gamma_n^{(2)} \setminus \Gamma_n$, we
have $C \,\not\subset\, D_2'$ by Lemma \ref{containment}. From
\eqref{linequiv} we know that $D_2' \,\equiv\, D_2 - c_{1,2} D_1$ for
a positive integer $c_{1,2}$. Hence $0 \,\le\, C \cdot D_2' \,=\,
p_2-c_{1,2}p_1$. This in turn implies $p_2 \,\ge\, p_1$. 

On the other hand, since $C \,\not\subset\, X_n^{(2)}$, we have $p_1 \,\ge\,
m$ (see the proof of Proposition \ref{gen-ineq}). So
$$\frac{L\cdot C}{m} \,=\, \frac{a_1p_1+a_2p_2+\ldots+a_np_n}{m} \,\ge\,
a_2.$$ Hence the lemma is proved. 
\end{proof}

So $\frac{L\cdot C}{m} \,\ge\, a_2 \,\ge\,
\varepsilon(X_n^{(2)},\,L|_{X_n^{(2)}},\,x)$
for all curve $C \,\not\subset\, X_n^{(2)}$. Then, from Lemma
\ref{lemma} it follows that 
$\varepsilon(L,x) \,=\, \varepsilon(X_n^{(2)},L|_{X_n^{(2)}},x)$, as
required. 

\textbf{Case 2}: $x\,\notin\, \Gamma_n^{(2)}$. 

In this case, $\varepsilon(X_n^{(2)},\,L|_{X_n^{(2)}},\,x)\, =\,
\varepsilon(X_n^{(3)},L|_{X_n^{(3)}},x)$. Now choose the smallest
integer $i$
such that $x \, \in\, \Gamma_n^{(i)}$. 
Note that such an $i$ exists, since $x \in \Gamma_n^{(n)} = X_n^{(n)}$. 
It follows that 
$3 \,\le\, i \,\le\, n$ and $x \,\notin\, \Gamma_n^{(j)}$ for
$j \,\le\, i-1$.

By the induction hypothesis,
$$\varepsilon(X_n^{(2)},\,L|_{X_n^{(2)}},\,x) \,=\,
\varepsilon(X_n^{(3)},\,L|_{X_n^{(3)}},\,x) \,= \,\cdots \,=\,
\varepsilon(X_n^{(i)},\,L|_{X_n^{(i)}},\,x). $$

Note that $L|_{X_n^{(i)}} \,=\, (a_i,\,\cdots,\,a_n)$ by Lemma \ref{resfibre}.
Again, by the induction hypothesis,   
$$\varepsilon(X_n^{(i)},\,L|_{X_n^{(i)}},\,x) \,=\,
\text{min}\{a_i,\,\varepsilon(X_n^{(i+1)},\,L|_{X_n^{(i+1)}},\,x)\},\ \text{~if~}\, i \,\le\,
n-1$$ $$\text{and~}\ \varepsilon(X_n^{(i)},\,L|_{X_n^{(i)}},\,x) \,=\, a_i,\ \text{~if~}\, i\,=\,n.$$ 
In either case, we get that $\varepsilon(X_n^{(i)},\,L|_{X_n^{(i)}},\,x)\, \le\, a_i$.

Hence $\varepsilon(X_n^{(2)},\,L|_{X_n^{(2)}},\,x) \,=\,
\,\cdots\, =
\varepsilon(X_n^{(i)},\,L|_{X_n^{(i)}},\,x) \,\le\, a_i$.

\begin{lemma} Let $C \,\subset\, X_n$ be an irreducible and reduced
curve such that $C \,\not\subset \,X_n^{(2)}$ and $m\,:=\, \text{mult}_x(C) \,>\,
0$. Then $\frac{L\cdot C}{m} \,\ge\, a_i$. 
\end{lemma}
\begin{proof}
Write $C \,=\, p_1 \Gamma_n + p_2\Gamma_n^{(2)}+\ldots + p_n\Gamma_n^{(n)}$ for some non-negative
integers $p_1,\, \cdots ,\, p_n$. 

Since $x \,\in\, \Gamma_n^{(i)} \setminus \Gamma_n^{(i-1)}$, we
have $C \,\not\subset \,D_i'$ by Lemma \ref{containment}. Further, 
$$D_i' \,\equiv\, D_i - c_{1,i} D_1-c_{2,i}D_2 - \ldots - c_{i-1,i}D_{i-1}$$ for
positive integers $c_{1,i}, \,\cdots,\, c_{i-1,i}$. Hence $$0 \,\le\, C \cdot D_i' \,=\, p_i-c_{1,i}p_1
-\ldots - c_{i-1,i}p_{i-1}$$ and this gives $p_i \,\ge\, p_1$.

On the other hand, since $C \,\not\subset\, X_n^{(2)}$, we have that $p_1 \,\ge\,
m$ (see the proof of Proposition \ref{gen-ineq}). So
$$\frac{L\cdot C}{m} \,=\, \frac{a_1p_1+a_2p_2+\ldots+a_np_n}{m} \,\ge\,
a_i.$$ Hence the lemma is proved. 
\end{proof}

Consequently, $\frac{L\cdot C}{m} \,\ge\, a_i \,\ge\,
\varepsilon(X_n^{(2)},\,L|_{X_n^{(2)}},\,x)$
for all curves $C \,\not\subset \,X_n^{(2)}$. Then, by Lemma
\ref{lemma} it follows that 
$\varepsilon(L,\,x)\, =\, \varepsilon(X_n^{(2)},\,L|_{X_n^{(2)}},\,x)$, as required.
This completes the proof of the proposition.
\end{proof}

\begin{proof}[Proof of Theorem \ref{main}]
The theorem follows immediately from Proposition \ref{case1} and Proposition \ref{case2}.
\end{proof}

\begin{remark}
A Bott tower of height two is a Hirzebruch surface (a geometrically 
ruled surface over $\mathbb{P}^1$) and Seshadri constants for ample line
bundles on such surfaces were computed in
\cite[Theorem 3.27]{Sy} and \cite[Theorem 4.1]{Ga}. 
When $n\,=\,2$, Theorem \ref{main} recovers the results of
\cite{Sy,Ga}. 
\end{remark}

\begin{corollary}\label{cor1}
For $n > 0$, let $X_n \to X_{n-1} \to \ldots \to X_2 \to X_1 \to X_0$ be a Bott
tower with positive Bott numbers. Let $L \equiv a_1 D_1 + \ldots + a_n
D_n$ be a nef line bundle
on $X_n$. Let $x \in X_n$. Then the following hold:
\begin{enumerate}
\item $\varepsilon(X_n,\,L,\,x) \,= \,{\rm min}\left\{L \cdot \Gamma_n^{(i)}
\,\mid\, x \in \Gamma_n^{(i)}\right\} = {\rm min}\left\{ a_i \, \mid \, x \in D_{i+1}' \cap \ldots \cap D_n'\right\}$. 
In
particular, $\varepsilon(X_n,\,L,x) \,= \,a_i$ for some $i$. 

\item $\varepsilon(X_n,\,L,x) \,\ge\, {\rm min}\{a_1,\,\cdots,\,a_n\}$. 

\item $\varepsilon(X_n,\,L,x) \,\le\, a_n$. 
\end{enumerate}

\end{corollary}
\begin{proof}
The statements in the corollary are derived from Theorem \ref{main}. If $x \,\notin \,\Gamma_n$, then 
$\varepsilon(X,\,L,x) \,=\, \varepsilon(X_n^{(2)},\,L|_{X_n^{(2)}},\,x)$ and the 
statements follow immediately from induction on $n$.

If $x \,\in\, \Gamma_n$, then $\varepsilon(X,\,L,\,x) \,= \,{\rm
min}\{a_1,\,\varepsilon(X_n^{(2)},\,L|_{X_n^{(2)}},\,x)\}$. It is easy to
see that the statements in the corollary hold for $\varepsilon(X_n,\,L,\,x)$ if they hold for 
$\varepsilon(X_n^{(2)},\,L|_{X_n^{(2)}},\,x)$.
\end{proof}

\begin{remark}\label{seshadri-curve}
Let $L$ be any nef line bundle  on a surface $X$ and let $x \in
X$. If $\varepsilon(X,L,x) \,=\, \frac{L\cdot C}{{\rm mult}_{x}C},$ 
then $C$ is said to be a \textit{Seshadri curve} for $L$ at $x$. 
If $X_n$ is a Bott tower and $L$ is a nef line bundle on $X_n$,
then the first statement of Corollary \ref{cor1} shows that
one of the curves 
$\Gamma_n,\, \Gamma_n^{(2)},\,\cdots,\,\Gamma_n^{(n)}$ (defined in \eqref{gni})
is a Seshadri curve for $L$ at any point $x \in X_n$. 
\end{remark}

\begin{corollary}\label{cor2}
For $n > 0$, let $X_n \to X_{n-1} \to \ldots \to X_2 \to X_1 \to X_0$ be a Bott
tower with positive Bott numbers. Let $L \,\equiv\, a_1 D_1 + \ldots + a_n
D_n$ be a nef line bundle
on $X_n$. Then 
\begin{enumerate}
\item $\varepsilon(X_n,\,L) \,=\, {\rm min}\{a_1,\,\cdots,\,a_n\}$, and

\item $\varepsilon(X_n,\,L,1)\, =\, a_n$. 
\end{enumerate}
\end{corollary}
\begin{proof}
We first prove (1). Let $a_i \,=\, {\rm min}\{a_1,\,\cdots,\,a_n\}$. 
It is easy to choose a point $x \,\in\, X_n$ such that 
$x \,\in\, \Gamma_n^{(i)} \setminus \Gamma_n^{(i-1)}$; for example, we can
take \(x=[z^0_1:w^0_1: \ldots: z^0_n:w^0_n]\) such that \(z^0_i\neq
0\) and \(z^0_l=0, w^0_l=1\) for \(l>i\). 
Then $x
\notin \Gamma_n^{(j)}$ for any $j \le i-1$. By Theorem \ref{main}, 
$$\varepsilon(L,\,x) \,=\, \varepsilon(L|_{X_n^{(2)}},\,x)\, =\, \ldots\, =\, 
\varepsilon(L|_{X_n^{(i)}},\,x)\, =\, \text{min}\{a_i,\,
\varepsilon(L|_{X_n^{(i+1)}},\,x)\}\,=\, a_i\, .$$ For the last equality, use
Corollary \ref{cor1}(1) and $a_i \,= \,{\rm min}\{a_1,\,\cdots,\,a_n\}$. 
Now it follows (again, using Corollary \ref{cor1}(1)) that the 
smallest Seshadri constant of $L$ is $a_i$. Hence
$\varepsilon(X_n,\,L)\,= \,a_i$. 

To prove (2), choose a point $x \in X_n$ satisfying $x \notin \Gamma_n^{(n-1)}$. Then 
$$\varepsilon(L,\,x) \,=\, \varepsilon(L|_{X_n^{(2)}},\,x) \,=\, \ldots\, =\, 
\varepsilon(L|_{X_n^{(n)}},\,x)
\,=\,\varepsilon(\mathbb{P}^1,\,\mathcal{O}_{\mathbb{P}^1}(a_n),\,x) \,=\,a_n\, .$$
Since all Seshadri constants are bounded above by $a_n$ and the value
$a_n$ is achieved at some point, it follows that $\varepsilon(L,1) \,=\, a_n$.
\end{proof}

\begin{remark}
From the proof of Corollary \ref{cor2}, we observe that 
\(\varepsilon(L,1) = \varepsilon(L,x)=a_n\) for any \(x \notin
\Gamma_n^{(n-1)}\). It is easy to see that for a general point $x \in
X_n$, we have $x \notin \Gamma_n^{(n-1)}$. 
Thus $\varepsilon(L,1)$ is achieved at general
points of $X_n$. On the other hand, $\varepsilon(X,L)$ is achieved at
\textit{special} points of $X_n$; namely, points $x$ satisfying $\Gamma_n^{(i)} \setminus
\Gamma_n^{(i-1)}$ if $a_i = {\rm min}\{a_1,\ldots,a_n\}$. 
\end{remark}

\subsection{Remarks and Examples}\label{examples}
Now we will compare our results with some other results on Seshadri
constants on toric varieties in literature. We will then give some examples to
illustrate our results. 

\begin{remark}\label{hmp} Seshadri constants for line bundles on toric
  varieties at torus fixed points are investigated in \cite{DiRocco} 
via generation of  jets. 
The case of equivariant vector
  bundles on toric varieties is studied in \cite{HMP}. At a torus fixed
  point $x$, the Seshadri constant of an equivariant vector bundle $E$
  is computed via the restriction of $E$ to the invariant curves passing through $x$
  (see \cite[Proposition 3.2]{HMP}). We show now that Corollary
\ref{cor1}(1) recovers this result for line bundles. So for line bundles, 
Corollary \ref{cor1}(1) can be viewed
as a generalization of these results for all points on a Bott tower. 
	
	Recall that there is one-one correspondence between the set of
        torus-fixed points in a smooth complete toric variety 
and the set of maximal cones. Let us denote the fixed point
corresponding to a maximal cone \(\sigma\) by \(x_{\sigma}\). 
Let $L \equiv a_1 D_1 + \ldots + a_n D_n$ be a nef line bundle
	on $X_n$. By \cite[Corollary 3.3]{HMP}, 
        $\varepsilon(L) = \min\limits_{x_{\sigma} } \varepsilon(L,
        x_{\sigma}),$ 
where the minimum varies over all maximal cones $\sigma $ in
$\Delta_n$. Now consider the maximal cone 
$\sigma=\text{Cone}(v_1, \ldots, v_n)$. Then by \cite[Corollary
4.2.2]{B}, we have \(\varepsilon(L, x_{\sigma})={\rm
  min}\{a_1,\ldots,a_n\}\), since the invariant curves passing through
\(x_{\sigma}\) are \(V(\tau_i)\), where $\tau_i=\text{Cone}(v_1,
\ldots, \widehat{v}_i, \ldots, v_n)$ for \(i=1, \ldots, n\) (see also  \cite[Proposition
3.2]{HMP}). Now
consider a 
maximal cone $\sigma'$ other that $\sigma$. Let \(C\) be an invariant
curve passing through \(x_{\sigma'}\). 
Then \(C=V(\tau)\) for an \((n-1)\)-dimensional cone $\tau$ of the form
	\begin{center}
		$\tau=\text{Cone}(v_1, \ldots, \widehat{v}_{i_\textsubscript{1}}, \ldots, \widehat{v}_{i_\textsubscript{\(r\)}}, \ldots, v_n,v_{n+i_\textsubscript{1}}, v_{n+i_\textsubscript{\(2\)}}, \ldots, \widehat{v}_{n+i_\textsubscript{\(j\)}}, \ldots, v_{n+i_\textsubscript{\(r\)}})$ 
	\end{center}
	for some $j=1, \ldots, r$, where \(r\) varies from \(1\) to
        \(n\) and $D \cdot V(\tau) \geq {\rm min}\{a_1,\ldots,a_n\}$
  (see \cite[proof of Theorem 3.1.1]{KD}). Thus 
  \(\varepsilon(L, x_{\sigma'}) \geq {\rm
    min}\{a_1,\ldots,a_n\}\). Hence 	
$\varepsilon(L) = {\rm min}\{a_1,\ldots,a_n\}$, which agrees with Corollary \ref{cor2} (1).
\end{remark}

\begin{remark}\label{non-integral}
Theorem \ref{main} shows that the Seshadri constants of ample line
bundles on Bott towers are integers at all points. By \cite[Corollary 4.2.2]{B}, this holds for
torus fixed points on an arbitrary toric variety. Note however that
Seshadri constants can be non-integral on an arbitrary toric
variety; see \cite[Example 1.4]{It2}. This example describes an ample
line bundle $L$ on a toric surface $X$ such that $L^2 = 3$ and $3/2
\le \varepsilon(X,L,x) \le \sqrt{3}$ for some point $x \in X$. In fact, in this example $X$ is a cubic surface in $\mathbb{P}^3$ and $L = \mathcal{O}_X(1)$. Let $x \in X$ be a general point. By considering a hyperplane $H \subset \mathbb{P}^3$ 
tangent to $X$ at $x$ and taking $ C = H \cap X$, we obtain an equality 
$\varepsilon(X,L,x) = 3/2$. 
See also \cite[Example 2.1]{ST}). 
\end{remark}

\begin{remark}\label{ito} 
In \cite{It2}, the author gives bounds on Seshadri constants on an
arbitrary toric variety at any point. In some cases, these bounds give
exact values. 
To apply this in our situation, let 
$L \equiv a_1 D_1 + \ldots + a_n D_n$ be an ample line bundle
	on $X_n$. Let \(x \in T \), the torus of \(X_n\). By a repeated
application of \cite[Theorem 3.6]{It2}, it is possible to show that 
\(\varepsilon(L,x)=a_n\). This is a special case of our results; for
example, it follows from Corollary \ref{cor1}, since  clearly 
\(x \notin \Gamma_n^{(n-1)}\). 
\end{remark}

We now give some examples illustrating our main theorem. We use the
same set-up as in Theorem \ref{main}. Note that in each example below
Corollary \ref{cor2} is verified. 

\begin{example} Let $L \equiv (1,3,8,4) \in \text{Pic}(X_4)$ and \(x
  \in X_4\). We repeatedly use Theorem \ref{main} and 
Corollary \ref{cor1} to compute the Seshadri constants of $L$. If \(x \in
\Gamma_4\) then $\varepsilon(L,x) = 1$. So 
assume now that \(x \notin \Gamma_4\). Then $\varepsilon(L,x) =
\varepsilon(L|_{X^{(2)}_4}, x)$. Note that 
\(L|_{X^{(2)}_4} \equiv (3,8,4)\). If \( x \in \Gamma_4^{(2)}\) then
$\varepsilon(L,x) = 3$. Finally, 
if \(x \notin \Gamma_4^{(2)}\), then $\varepsilon(L,x) =\varepsilon(L|_{X^{(3)}_4}, x)=4$.
	
	Thus   \begin{equation*}
		\varepsilon(L,x) = 
		\begin{cases}
			1, & {\rm if}\ x \in \Gamma_4, \\
			3,  & {\rm if}\ x \notin \Gamma_4, x \in \Gamma_4^{(2)},\\
			4, & {\rm if}\ x \notin \Gamma_4, x \notin \Gamma_4^{(2)}.
		\end{cases}
	\end{equation*} 
\end{example}

\begin{example} Let $L \equiv (1,2,3,8) \in \text{Pic}(X_4)$ and \(x
  \in X_4\). Repeatedly applying Theorem \ref{main} and Corollary
  \ref{cor1}, 
	
	   \begin{equation*}
		\varepsilon(L,x) = 
		\begin{cases}
			1, & {\rm if}\ x \in \Gamma_4, \\
			2,  & {\rm if}\ x \notin \Gamma_4, x \in \Gamma_4^{(2)},\\
			3, & {\rm if}\ x \notin \Gamma_4, x \notin \Gamma_4^{(2)}, x \in \Gamma_4^{(3)},\\
			8, &{\rm if}\ x \notin \Gamma_4, x \notin \Gamma_4^{(2)}, x \notin \Gamma_4^{(3)}.
		\end{cases}
	\end{equation*} 
\end{example}

\begin{example} Let $L \equiv (3,6,2,7) \in \text{Pic}(X_4)$ and \(x \in X_4\). Here note that \(x \in \Gamma_4 \Rightarrow x \in \Gamma_4^{(2)} \Rightarrow x \in \Gamma_4^{(3)}\).
	
	Then   \begin{equation*}
		\varepsilon(L,x) = 
		\begin{cases}
			2,  & {\rm if}\ x \in \Gamma_4, \\
			2, & {\rm if}\ x \notin \Gamma_4,  x \in \Gamma_4^{(3)},\\
			7, &{\rm if}\ x \notin \Gamma_4,  x \notin \Gamma_4^{(3)}.
		\end{cases}
	\end{equation*} 
\end{example}

\begin{example} Let $L \equiv (3,6,5,7,9) \in \text{Pic}(X_5)$ and \(x \in X_5\). Here note that \(x \in \Gamma_5 \Rightarrow x \in \Gamma_5^{(2)} \Rightarrow x \in \Gamma_5^{(3)} \Rightarrow x \in \Gamma_5^{(4)}\).
	
	Then   \begin{equation*}
		\varepsilon(L,x) = 
		\begin{cases}
			3,  & {\rm if}\ x \in \Gamma_5, \\
                        5, & {\rm if}\ x \notin \Gamma_5,  x \in \Gamma_5^{(2)},\\
			5, & {\rm if}\ x \notin \Gamma_5,  x \notin
                        \Gamma_5^{(2)}, x \in \Gamma_5^{(3)}, \\
			7, & {\rm if}\ x \notin \Gamma_5,  x \notin
                        \Gamma_5^{(2)}, x \notin \Gamma_5^{(3)}, x \in
                        \Gamma_5^{(4)}, \\
                        9, & {\rm if}\ x \notin \Gamma_5,  x \notin
                        \Gamma_5^{(2)}, x \notin \Gamma_5^{(3)}, x \notin
                        \Gamma_5^{(4)}.
		\end{cases}
	\end{equation*} 
\end{example}


\end{document}